\documentclass[12pt]{article}
\usepackage{amsmath,amsfonts,amssymb,amsthm}

\theoremstyle{plain}
\newtheorem{theorem}{Theorem}[section]
\newtheorem{corollary}[theorem]{Corollary}
\newtheorem{lemma}[theorem]{Lemma}
\newtheorem{proposition}[theorem]{Proposition}

\theoremstyle{definition}

\newtheorem{example}[theorem]{Example}

\numberwithin{equation}{section}

\newcommand{\R}{{\mathbb R}}
\newcommand{\N}{{\mathbb N}}

\newcommand{\Om}{\Omega}

\providecommand{\vint}[1]{\mathchoice
          {\mathop{\vrule width 5pt height 3 pt depth -2.5pt
                  \kern -9pt \kern 1pt\intop}\nolimits_{\kern -5pt{#1}}}
          {\mathop{\vrule width 5pt height 3 pt depth -2.6pt
                  \kern -6pt \intop}\nolimits_{\kern -3pt{#1}}}
          {\mathop{\vrule width 5pt height 3 pt depth -2.6pt
                  \kern -6pt \intop}\nolimits_{\kern -3pt{#1}}}
          {\mathop{\vrule width 5pt height 3 pt depth -2.6pt
                  \kern -6pt \intop}\nolimits_{\kern -3pt{#1}}}}

\newcommand{\eps}{\varepsilon}
\newcommand{\loc}{\mathrm{loc}}

\newcommand{\BV}{\mathrm{BV}}
\newcommand{\liploc}{\mathrm{Lip}_{\mathrm{loc}}}

\newcommand{\ch}{\text{\raise 1.3pt \hbox{$\chi$}\kern-0.2pt}}
\newcommand{\niton}{\not\owns}

\DeclareMathOperator{\capa}{Cap}

\DeclareMathOperator{\dist}{dist}

\DeclareMathOperator{\Lip}{Lip}

\def\XXint#1#2#3{{\setbox0=\hbox{$#1{#2#3}{\int}$}
\vcenter{\hbox{$#2#3$}}\kern-.5\wd0}}

\begin{document}
\title{Strict and pointwise convergence of
\\ $\BV$ functions in metric spaces
\footnote{{\bf 2010 Mathematics Subject Classification}: 30L99, 26B30, 28A20.
\hfill \break {\it Keywords\,}: metric measure space, bounded variation,
strict convergence, pointwise convergence, uniform convergence, codimension one Hausdorff measure
}}
\author{Panu Lahti}
\maketitle

\begin{abstract}
In the setting of a metric space $X$ equipped with a doubling measure that supports a Poincar\'e inequality, we show that if $u_i\to u$ strictly in $\BV(X)$, i.e. if $u_i\to u$ in $L^1(X)$ and $\Vert Du_i\Vert(X)\to\Vert Du\Vert(X)$, then for a subsequence (not relabeled) we have $\widetilde{u}_i(x)\to \widetilde{u}(x)$ for $\mathcal H$-almost every $x\in X\setminus S_u$.
\end{abstract}

\section{Introduction}

Let $(X,d,\mu)$ be a metric space equipped with a doubling measure $\mu$ that supports a  $(1,1)$-Poincar\'e inequality. Take a sequence of $\BV$ functions $(u_i)\subset \BV(X)$.
If $u_i\to u$ in $L^1(X)$, then of course for a subsequence (not relabeled) we have $u_i(x)\to u(x)$ for $\mu$-almost every $x\in X$.
If $X=\R^n$ and the functions $u_i$ are defined as convolutions of $u\in\BV(\R^n)$ with a mollifier function at smaller and smaller scales,
then $u_i(x)\to \widetilde{u}(x)$ for $\mathcal H$-almost every
$x\in \R^n$ by \cite[Corollary 3.80]{AFP} (where $\mathcal H$ is the codimension one, or $n-1$-dimensional Hausdorff measure; see Section \ref{sec:preliminaries} for notation).  See also \cite[Proposition 4.1]{KKST3}
for a slightly weaker analogous result in the metric setting.

In this paper we consider what kind of pointwise convergence can be obtained if
we know that $u_i\to u$ \emph{strictly} in $\BV(\Omega)$, that is, $u_i\to u$ in $L^1(\Omega)$ and $\Vert Du_i\Vert(\Omega)\to \Vert Du\Vert(\Omega)$,
where $\Om\subset X$ is an open set.
We show that in this case, for a subsequence (not relabeled) we have $\widetilde{u}_i(x)\to \widetilde{u}(x)$ for $\mathcal H$-almost every $x\in \Omega\setminus S_u$. This is given in Corollary \ref{cor:strict to pointwise convergence}.
We also show that in any compact subset of $\Omega\setminus S_u$, we can obtain uniform convergence outside sets of small $1$-capacity.
This is given in Corollary \ref{cor:uniform convergence outside jump set}.
Somewhat more general formulations of these results are given in Theorem \ref{thm:strict to pointwise convergence} and Theorem \ref{thm:uniform convergence outside jump set}.

Our results seem to be new even in the Euclidean setting.
The main tool used in the proofs is a boxing inequality -type argument, which has been previously applied in e.g. \cite{KKST2}.
At the end of the paper we give examples that demonstrate that the results appear to be
optimal.

\section{Preliminaries}\label{sec:preliminaries}

In this section we introduce the necessary notation and assumptions.

In this paper, $(X,d,\mu)$ is a complete metric space equipped
with a Borel regular outer measure $\mu$ satisfying a doubling property, that is,
there is a constant $C_d\ge 1$ such that
\[
0<\mu(B(x,2r))\leq C_d\,\mu(B(x,r))<\infty
\]
for every ball $B=B(x,r)$ with center $x\in X$ and radius $r>0$.
By iterating the doubling property, we obtain that for any $x\in X$ and $y\in B(x,R)$ with $0<r\le R<\infty$, we have
\begin{equation}\label{eq:homogenous dimension}
\frac{\mu(B(y,r))}{\mu(B(x,R))}\ge \frac{1}{C_d^2}\left(\frac{r}{R}\right)^{Q},
\end{equation}
where $Q>0$ only depends on the doubling constant $C_d$. When we want to specify that a constant $C$
depends on the parameters $a,b, \ldots,$ we write $C=C(a,b,\ldots)$.

A complete metric space with a doubling measure is proper,
that is, closed and bounded sets are compact. Since $X$ is proper, for any open set $\Omega\subset X$
we define $\liploc(\Omega)$ to be the space of
functions that are Lipschitz in every $\Omega'\Subset\Omega$.
Here $\Omega'\Subset\Omega$ means that $\Omega'$ is open and that $\overline{\Omega'}$ is a
compact subset of $\Omega$. Other local spaces of functions are defined similarly.

For any set $A\subset X$ and $0<R<\infty$, the Hausdorff content
of codimension one is defined by
\[
\mathcal{H}_{R}(A):=\inf\left\{ \sum_{i=1}^{\infty}
  \frac{\mu(B(x_{i},r_{i}))}{r_{i}}:\,A\subset\bigcup_{i=1}^{\infty}B(x_{i},r_{i}),\,r_{i}\le R\right\}.
\]
The codimension one Hausdorff measure of a set $A\subset X$ is given by
\[
\mathcal{H}(A):=\lim_{R\rightarrow 0}\mathcal{H}_{R}(A).
\]

The measure theoretic boundary $\partial^{*}E$ of a set $E\subset X$ is the set of points $x\in X$
at which both $E$ and its complement have positive upper density, i.e.
\[
\limsup_{r\to 0}\frac{\mu(B(x,r)\cap E)}{\mu(B(x,r))}>0\quad\;
  \textrm{and}\quad\;\limsup_{r\to 0}\frac{\mu(B(x,r)\setminus E)}{\mu(B(x,r))}>0.
\]
The measure theoretic interior and exterior of $E$ are defined respectively by
\[
I_E:=\left\{x\in X:\,\lim_{r\to 0}\frac{\mu(B(x,r)\setminus E)}{\mu(B(x,r))}=0\right\}
\]
and
\[
O_E:=\left\{x\in X:\,\lim_{r\to 0}\frac{\mu(B(x,r)\cap E)}{\mu(B(x,r))}=0\right\}.
\]
Note that the space is always partitioned into the disjoint sets $\partial^*E$, $I_E$, and
$O_E$.

A curve $\gamma$ is a rectifiable continuous mapping from a compact interval
into $X$.
A nonnegative Borel function $g$ on $X$ is an upper gradient 
of an extended real-valued function $u$
on $X$ if for all curves $\gamma$ on $X$, we have
\[
|u(x)-u(y)|\le \int_\gamma g\,ds,
\]
where $x$ and $y$ are the end points of $\gamma$. We interpret $|u(x)-u(y)|=\infty$ whenever  
at least one of $|u(x)|$, $|u(y)|$ is infinite. Upper gradients were originally introduced in~\cite{HK}.

We consider the following norm
\[
\Vert u\Vert_{N^{1,1}(X)}:=\Vert u\Vert_{L^1(X)}+\inf \Vert g\Vert_{L^1(X)},
\]
where the infimum is taken over all upper gradients $g$ of $u$.
The substitute for the Sobolev space $W^{1,1}(X)$ in the metric setting is the Newton-Sobolev space
\[
N^{1,1}(X):=\{u:\|u\|_{N^{1,1}(X)}<\infty\}.
\]
For more on Newton-Sobolev spaces, we refer to~\cite{S, BB}.

The $1$-capacity of a set $A\subset X$ is given by
\begin{equation}\label{eq:definition of p-capacity}
 \capa_1(A):=\inf \Vert u\Vert_{N^{1,1}(X)},
\end{equation}
where the infimum is taken over all functions $u\in N^{1,1}(X)$ such that $u\ge 1$ in $A$.
We know that $\capa_1$ is an outer capacity, meaning that
\[
\capa_1(A)=\inf\{\capa_1(U):\,U\supset A\textrm{ is open}\}
\]
for any $A\subset X$, see e.g. \cite[Theorem 5.31]{BB}.

Next we recall the definition and basic properties of functions
of bounded variation on metric spaces, see \cite{M}. See also e.g. \cite{AFP, EvaG92, Giu84, Zie89} for the classical 
theory in the Euclidean setting.
For $u\in L^1_{\loc}(X)$, we define the total variation of $u$ in $X$ to be 
\[
\|Du\|(X):=\inf\left\{\liminf_{i\to\infty}\int_X g_{u_i}\,d\mu:\, u_i\in \Lip_{\loc}(X),\, u_i\to u\textrm{ in } L^1_{\loc}(X)\right\},
\]
where each $g_{u_i}$ is an upper gradient of $u_i$.
We say that a function $u\in L^1(X)$ is of bounded variation, 
and denote $u\in\BV(X)$, if $\|Du\|(X)<\infty$.
By replacing $X$ with an open set $\Omega\subset X$ in the definition of the total variation, we can define $\|Du\|(\Omega)$.
For an arbitrary set $A\subset X$, we define
\[
\|Du\|(A)=\inf\{\|Du\|(\Omega):\, A\subset\Omega,\,\Omega\subset X
\text{ is open}\}.
\]
If $u\in\BV(\Omega)$, $\|Du\|(\cdot)$ is a finite Radon measure on $\Omega$ by \cite[Theorem 3.4]{M}.
The $\BV$ norm is defined by
\[
\Vert u\Vert_{\BV(\Omega)}:=\Vert u\Vert_{L^1(\Omega)}+\Vert Du\Vert(\Omega).
\]
A $\mu$-measurable set $E\subset X$ is said to be of finite perimeter if $\|D\ch_E\|(X)<\infty$, where $\ch_E$ is the characteristic function of $E$.
The perimeter of $E$ in $\Omega$ is also denoted by
\[
P(E,\Omega):=\|D\ch_E\|(\Omega).
\]
Similarly as above, if $P(E,\Omega)<\infty$, then $P(E,\cdot)$ is a finite Radon measure
on $\Omega$.
For any Borel sets $E_1,E_2\subset X$ we have by \cite[Proposition 4.7]{M}
\[
P(E_1\cap E_2,X)+P(E_1\cup E_2,X)\le P(E_1,X)+P(E_2,X). 
\]
The proof works equally well for $\mu$-measurable $E_1,E_2\subset X$ and with $X$
replaced by any open set $U\subset X$, so that
\begin{equation}\label{eq:Caccioppoli sets form an algebra}
P(E_1\cap E_2,U)+P(E_1\cup E_2,U)\le P(E_1,U)+P(E_2,U). 
\end{equation}
We have the following coarea formula from~\cite[Proposition 4.2]{M}: if $\Omega\subset X$ is an open set and $v\in L^1_{\loc}(\Omega)$, then
\begin{equation}\label{eq:coarea}
\|Dv\|(\Omega)=\int_{-\infty}^{\infty}P(\{v>t\},\Omega)\,dt.
\end{equation}

We will assume throughout that $X$ supports a $(1,1)$-Poincar\'e inequality,
meaning that there exist constants $C_P\ge 1$ and $\lambda \ge 1$ such that for every
ball $B(x,r)$, every $u\in L^1_{\loc}(X)$,
and every upper gradient $g$ of $u$, we have 
\[
\vint{B(x,r)}|u-u_{B(x,r)}|\, d\mu 
\le C_P r\vint{B(x,\lambda r)}g\,d\mu,
\]
where 
\[
u_{B(x,r)}:=\vint{B(x,r)}u\,d\mu :=\frac 1{\mu(B(x,r))}\int_{B(x,r)}u\,d\mu.
\]
By applying the Poincar\'e inequality to approximating locally Lipschitz functions in the definition of the total variation, 
we get the following $(1,1)$-Poincar\'e inequality for $\BV$ functions.
For every ball $B(x,r)$ and every 
$u\in L^1_{\loc}(X)$, we have
\[
\vint{B(x,r)}|u-u_{B(x,r)}|\,d\mu
\le C_P r\, \frac{\Vert Du\Vert (B(x,\lambda r))}{\mu(B(x,\lambda r))}.
\]
For a $\mu$-measurable set $E\subset X$, the above implies the relative isoperimetric inequality
\begin{equation}\label{eq:relative isoperimetric inequality}
\min\{\mu(B(x,r)\cap E),\,\mu(B(x,r)\setminus E)\}\le 2 C_P rP(E,B(x,\lambda r)).
\end{equation}

Given a set of finite perimeter $E\subset X$, for $\mathcal H$-almost every $x\in \partial^*E$ we have
\begin{equation}\label{eq:definition of gamma}
\gamma \le \liminf_{r\to 0} \frac{\mu(B(x,r)\cap E)}{\mu(B(x,r))} \le \limsup_{r\to 0} \frac{\mu(B(x,r)\cap E)}{\mu(B(x,r))}\le 1-\gamma,
\end{equation}
where $\gamma \in (0,1/2]$ only depends on the doubling constant and the constants in the Poincar\'e inequality, 
see~\cite[Theorem 5.4]{A1}.
For an open set $\Omega\subset X$ and a $\mu$-measurable set $E\subset X$
with $P(E,\Omega)<\infty$, we know that for any Borel set $A\subset \Omega$
\begin{equation}\label{eq:def of theta}
P(E,A)=\int_{\partial^{*}E\cap A}\theta_E\,d\mathcal H,
\end{equation}
where
$\theta_E\colon X\to [\alpha,C_d]$ with $\alpha=\alpha(C_d,C_P,\lambda)>0$, see \cite[Theorem 5.3]{A1} 
and \cite[Theorem 4.6]{AMP}.

The jump set of a $\mu$-measurable function $u$ on $X$ is defined by
\[
S_{u}:=\{x\in X:\, u^{\wedge}(x)<u^{\vee}(x)\},
\]
where $u^{\wedge}(x)$ and $u^{\vee}(x)$ are the lower and upper approximate limits of $u$ defined by
\[
u^{\wedge}(x):
=\sup\left\{t\in\overline\R:\,\lim_{r\to 0}\frac{\mu(B(x,r)\cap\{u<t\})}{\mu(B(x,r))}=0\right\}
\]
and
\[
u^{\vee}(x):
=\inf\left\{t\in\overline\R:\,\lim_{r\to 0}\frac{\mu(B(x,r)\cap\{u>t\})}{\mu(B(x,r))}=0\right\}.
\]

Note that for $u=\ch_E$, we have $x\in I_E$ if and only if $u^{\wedge}(x)=u^{\vee}(x)=1$, $x\in O_E$ if and only if $u^{\wedge}(x)=u^{\vee}(x)=0$, and $x\in \partial^*E$ if and only if $u^{\wedge}(x)=0$ and $u^{\vee}(x)=1$. 

We understand $\BV$ functions to be $\mu$-equivalence classes. To consider 
pointwise properties, we need to consider the representatives $u^{\wedge}$ and
$u^{\vee}$. We also define the representative
\[
\widetilde{u}:=(u^{\wedge}+u^{\vee})/2.
\]

We say that
a set $A\subset X$ is \emph{$1$-quasiopen} if for every $\eps>0$ there is an
open set $G\subset X$ with $\capa_1(G)<\eps$ such that $A\cup G$ is open.

\section{The convergence results}

In this section we give our main results: Theorem \ref{thm:strict to pointwise convergence}
on the pointwise convergence of $\BV$ functions, and Theorem \ref{thm:uniform convergence outside jump set}, given at the end of the section, on uniform convergence.

The following fact about the Hausdorff content and measure is well known in the Euclidean setting, and proved in the metric setting in the below reference. See also \cite[Lemma 7.9]{KKST2} for a previous similar result.

\begin{lemma}[{\cite[Lemma 3.7]{L2}}]\label{lem:Hausdorff content and measure}
Let $R>0$. If $A\subset X$ and $\mathcal H_R(A)=0$, then $\mathcal H(A)=0$.
\end{lemma}

\begin{theorem}\label{thm:strict to pointwise convergence}
Let $\Omega\subset X$ be an open set, and let $u_i,u\in \BV(\Omega)$ such that $u_i\to u$ in $L^1(\Omega)$ and $\Vert Du_i\Vert(\Omega)\to \Vert Du\Vert(\Omega)$. Then there exists a subsequence (not relabeled) such that for $\mathcal H$-almost every $x\in \Omega$,
\[
u^{\wedge}(x)\le\liminf_{i\to\infty}u_i^{\wedge}(x)
\le
\limsup_{i\to\infty}u_i^{\vee}(x)\le u^{\vee}(x).
\]
\end{theorem}

By the definitions of $u^{\wedge}$, $u^{\vee}$, and $\widetilde{u}$, we immediately get the following corollary.

\begin{corollary}\label{cor:strict to pointwise convergence}
Let $\Omega\subset X$ be an open set, and let $u_i,u\in \BV(\Omega)$ such that $u_i\to u$ in $L^1(\Omega)$ and $\Vert Du_i\Vert(\Omega)\to \Vert Du\Vert(\Omega)$. Then there exists a subsequence (not relabeled) such that
$\widetilde{u}_i(x)\to \widetilde{u}(x)$ for $\mathcal H$-almost every
$x\in \Omega\setminus S_u$.
\end{corollary}

First we prove the theorem for sets of finite perimeter. In the proof below, the definition of the sets $I_j$ and $O_j$ is inspired by the proof of Federer's structure theorem, see \cite[Section 4.5.11]{Fed} or \cite[p. 222]{EvaG92}.

\begin{proposition}\label{prop:strict to pointwise convergence for sets}
Let $\Omega\subset X$ be an open set, and let $E_i,E\subset X$ be $\mu$-measurable sets with  $\ch_{E_i}\to \ch_E$ in $L^1_{\loc}(\Omega)$
and $P(E_i,\Omega)\to P(E,\Omega)<\infty$.
Passing to a subsequence (not relabeled), for $\mathcal H$-almost every
$x\in I_E\cap\Omega$ we have $x\in I_{E_i}$ for all sufficiently large $i\in\N$, and for $\mathcal H$-almost every $x\in O_E\cap\Omega$ we have $x\in O_{E_i}$ for all sufficiently large $i\in\N$.
\end{proposition}

\begin{proof}
For each $j\in\N$, define (recall the number $\gamma$ from \eqref{eq:definition of gamma})
\[
I_j
:=\left\{ x\in X :\, \sup_{0<r<1/j}\frac{\mu(B(x,r)\setminus E)}{\mu(B(x,r))}\le\gamma/2\right\},
\]
\[
O_j
:=\left\{ x\in X :\, \sup_{0<r<1/j}\frac{\mu(B(x,r)\cap E)}{\mu(B(x,r))}\le\gamma/2\right\}.
\]
Note that these are increasing sequences of sets and
\begin{equation}\label{eq:IE and OE as unions}
I_E\subset \bigcup_{j=1}^{\infty} I_j\quad\text{and}\quad O_E\subset \bigcup_{j=1}^{\infty} O_j.
\end{equation}
Moreover, the sets $I_j$ and $O_j$ are closed, which can be seen as follows. For a fixed
$j\in\N$, take a sequence of points $x_k\in I_j$
with $x_k\to x\in X$. Let $0<r<1/j$. Then
by applying Lebesgue's dominated convergence theorem to both the numerator and
the denominator,
we obtain
\[
\frac{\mu(B(x,r)\setminus E)}{\mu(B(x,r))}=\lim_{k\to \infty}\frac{\mu(B(x_k,r-d(x_k,x)) \setminus E)}{\mu(B(x_k,r-d(x_k,x)))} \le \gamma/2,
\]
so $I_j$ is closed. The proof for the sets $O_j$ is analogous.
By \eqref{eq:definition of gamma} we also know that
$\mathcal H(\partial^*E\cap (I_j\cup O_j)\cap\Omega)=0$ for all $j\in\N$.
Then by \eqref{eq:def of theta} we have
\begin{equation}\label{eq:perimeter in Iks and Oks}
P\left(E,\bigcup_{j=1}^{\infty}(I_j\cup O_j)\cap\Omega\right)=0.
\end{equation}
We can find sets $\Omega'_j\Subset \Omega$ with $\Omega=\bigcup_{j=1}^{\infty}\Omega'_j$ such that
for a fixed $x_0\in X$, $\Omega'_j\subset B(x_0,j)$ for each $j\in\N$.
Fix $j\in\N$. By the lower semicontinuity of perimeter with respect to
$L^1$-convergence, we have for any open set $U\subset\Omega$
\[
P(E,U)\le \liminf_{i\to\infty} P(E_i,U).
\]
Since we also have $P(E_i,\Omega)\to P(E,\Omega)$, we get for any closed set $F\subset \Omega$
\[
\limsup_{i\to\infty} P(E_i,F)\le P(E,F).
\]
For any $A\subset X$ and $a>0$, denote
\[
A^{a}:=\{x\in X:\, \dist(x,A)<a\}.
\]
In particular, for any $0<a<\dist(\Omega_j',X\setminus \Omega)$ we have
\begin{equation}\label{eq:upper semicontinuity}
\limsup_{i\to\infty} P(E_i,(I_j\cap\Omega'_j)^a)\le\limsup_{i\to\infty}
P(E_i,\overline{(I_j\cap\Omega'_j)^a})\le P(E,\overline{(I_j\cap\Omega'_j)^a}).
\end{equation}
Using the fact that $P(E,\cdot)$ is a Radon measure on $\Omega$ in the second equality
below, and then the fact that $\overline{I_{j}\cap\Omega'_j}\subset I_j\cap\overline{\Omega'_j}$ since $I_j$ is closed,
\begin{align*}
\lim_{a\to 0}P(E,\overline{(I_j\cap\Omega'_j)^a})
&=\lim_{a\to 0}P(E,(I_j\cap\Omega'_j)^a)
=P\left(E,\bigcap_{a>0}(I_j\cap\Omega'_j)^a\right)\\
&=P(E,\overline{I_{j}\cap\Omega'_j})
\le P(E,I_j\cap\overline{\Omega'_j})=0
\end{align*}
by \eqref{eq:perimeter in Iks and Oks}.
Thus by choosing $0<s_j<\min\{1/5,\dist(\Omega'_j,X\setminus \Omega)\}$ small enough,
we get
\begin{equation}\label{eq:small perimeter in small neighborhood}
P(E,\overline{(I_j\cap\Omega'_j)^{s_j}})< 2^{-j}.
\end{equation}
Noting that $\Omega'_j\subset B(x_0,j)$, by \eqref{eq:homogenous dimension} we have 
\[
\inf_{y\in \Omega'}\mu(B(y,s_j/2\lambda))\ge \frac{1}{C_d^2}\left(\frac{s_j}{2\lambda j}\right)^Q\mu(B(x_0,j))>0.
\]
By \eqref{eq:upper semicontinuity} and \eqref{eq:small perimeter in small neighborhood},
 we find $i_j\in\N$ such that
\begin{equation}\label{eq:small perimeter in small neighborhood for Eijs}
P(E_{i_j},(I_j\cap\Omega'_j)^{s_j})< 2^{-j}
\end{equation} 
and (by the fact that $\ch_{E_i}\to\ch_E$ in $L^1_{\loc}(\Omega)$)
\begin{equation}\label{eq:smallness of L1 norm and measures of balls}
2C_d\Vert \ch_{E_{i_j}}-\ch_E\Vert_{L^1\left((\Omega_j')^{s_j}\right)}\le \inf_{y\in \Omega'}\mu(B(y,s_j/2\lambda)).
\end{equation}
Take $x\in I_{j}\cap O_{E_{i_j}}\cap \Omega'_j$. Since $E$ has lower density at least $1-\gamma/2$ at $x$ and $E_{i_j}$ has density zero, for some $0<s<s_j/2\lambda$ we have
in particular
\[
\vint{B(x,s)}|\ch_{E_{i_j}}-\ch_E|\,d\mu>\frac{1}{2C_d}.
\]
Next, double the radius $s$ repeatedly and take the last  number so obtained for which the above estimate holds, and call this number $r$. Note by \eqref{eq:smallness of L1 norm and measures of balls} that $r< s_j/2\lambda$. Then necessarily
\[
\frac{1}{2C_d}< \vint{B(x,r)}|\ch_{E_{i_j}}-\ch_E|\,d\mu\le \frac{1}{2},
\]
or equivalently
\[
\frac{1}{2C_d}< \frac{\mu(B(x,r)\cap (E_{i_j}\Delta E))}{\mu(B(x,r))}\le \frac{1}{2}.
\]
Then by the relative isoperimetric inequality \eqref{eq:relative isoperimetric inequality}
and \eqref{eq:Caccioppoli sets form an algebra}
\begin{equation}\label{eq:covering wrt perimeters of Eij and E}
\begin{split}
\frac{\mu(B(x,r))}{2C_d}
&\le \mu(B(x,r)\cap (E_{i_j}\Delta E))\\
&\le 2C_P r P(E_{i_j}\Delta E,B(x,\lambda r))\\
&\le 2C_P r \big(P(E_{i_j}\setminus E,B(x,\lambda r))+P(E\setminus E_{i_j},B(x,\lambda r))\big)\\
&\le 4 C_P r \big(P(E_{i_j},B(x,\lambda r))+P(E,B(x,\lambda r))\big).
\end{split}
\end{equation}
Using these radii, we obtain a covering
$\{B(x,\lambda r_x)\}_{x\in I_{j}\cap O_{E_{i_j}}\cap \Omega'_j}$
with $r_x< s_j/2\lambda$.
By the 5-covering theorem, we can extract a countable collection of pairwise disjoint balls
$\{B(x_m,\lambda r_m)\}_{m\in\N}$ with $r_m< s_j/\lambda$ such that the balls $\{B(x_m,5 \lambda r_m)\}_{m\in\N}$ cover $I_{j}\cap O_{E_{i_j}}\cap \Omega'_j$.
Denote by $\lceil a\rceil$ the smallest integer at least $a\in\R$.
Then
\begin{align*}
&\mathcal H_{5s_j}(I_{j}\cap O_{E_{i_j}}\cap \Omega'_j)
\le \sum_{m=1}^{\infty}\frac{\mu(B(x_m,5\lambda r_m))}{5\lambda r_m}\\
&\qquad\le C_d^{\lceil \log_2(5\lambda)\rceil}\sum_{m=1}^{\infty}\frac{\mu(B(x_m,r_m))}{r_m}\\
&\qquad\overset{\eqref{eq:covering wrt perimeters of Eij and E}}{\le} 8 C_d^{\lceil \log_2(5\lambda)\rceil+1} C_P \sum_{m=1}^{\infty} (P(E_{i_j},B(x_m,\lambda r_m))+P(E,B(x_m,\lambda r_m)))\\
&\qquad\le 8 C_d^{\lceil \log_2(5\lambda)\rceil+1} C_P (P(E_{i_j},(I_j\cap \Omega'_j)^{s_j})+P(E,(I_j\cap \Omega'_j)^{s_j}))\\
&\qquad \le 8 C_d^{\lceil \log_2(5\lambda)\rceil+1} C_P \times 2^{-j+1}
\end{align*}
by \eqref{eq:small perimeter in small neighborhood} and \eqref{eq:small perimeter in small neighborhood for Eijs}.
In total, using also \eqref{eq:def of theta},
\begin{align*}
&\mathcal H_{5s_j}(I_j\cap (\partial^* E_{i_j}\cup O_{E_{i_j}})\cap \Omega'_j)\\
&\qquad\quad \le \mathcal H_{5s_j}(I_j\cap \partial^* E_{i_j}\cap \Omega'_j)+\mathcal H_{5s_j}(I_j\cap O_{E_{i_j}}\cap \Omega'_j)\\
&\qquad\quad \le \alpha^{-1}P(E_{i_j}, I_{j}\cap\Omega'_j)+ 8 C_d^{\lceil \log_2(5\lambda)\rceil+1} C_P \times 2^{-j+1}\\
&\qquad\quad \le 2^{-j}\alpha^{-1} + 2^{-j+4} C_d^{\lceil \log_2(5\lambda)\rceil+1} C_P
\end{align*}
by \eqref{eq:small perimeter in small neighborhood for Eijs}.
Then for any $p\in\N$, since $I_p\subset I_j$ and $\Omega_p'\subset \Omega_j'$ as soon as $j\ge p$,
\begin{equation}\label{eq:main estimate for bad set}
\begin{split}
&\mathcal H_1\left(I_p\cap \Omega'_p\cap \bigcap_{l=1}^{\infty}\bigcup_{j=l}^{\infty}(\partial^* E_{i_j}\cup O_{E_{i_j}})\right)\\
&\qquad\qquad\le \lim_{l\to \infty} \mathcal H_1\left(I_p\cap \Omega'_p\cap \bigcup_{j=l}^{\infty}(\partial^* E_{i_j}\cup O_{E_{i_j}})\right)\\
&\qquad\qquad\le \lim_{l\to \infty}
\sum_{j=l}^{\infty}\mathcal H_{1}(I_p\cap\Omega'_p\cap (\partial^* E_{i_j}\cup O_{E_{i_j}}))\\
&\qquad\qquad\le \lim_{l\to \infty}
\sum_{j=l}^{\infty}\mathcal H_{5s_j}(I_j\cap\Omega'_j\cap (\partial^* E_{i_j}\cup O_{E_{i_j}}))\\
&\qquad\qquad\le\lim_{l\to \infty}( 2^{-l+1}\alpha^{-1} + 2^{-l+5} C_d^{\lceil \log_2(5\lambda)\rceil+1} C_P)\\
&\qquad\qquad =0.
\end{split}
\end{equation}
Thus by \eqref{eq:IE and OE as unions}
\begin{align*}
&\mathcal H_1\left(I_E\cap \Omega\cap\bigcap_{l=1}^{\infty}\bigcup_{j=l}^{\infty}(\partial^* E_{i_j}\cup O_{E_{i_j}})\right)\\
&\qquad\quad \le \sum_{p=1}^{\infty}\mathcal H_1\left(I_p\cap \Omega'_p\cap\bigcap_{l=1}^{\infty}\bigcup_{j=l}^{\infty}(\partial^* E_{i_j}\cup O_{E_{i_j}})\right)=0,
\end{align*}
and so by Lemma \ref{lem:Hausdorff content and measure},
\[
\mathcal H\left(I_E\cap \Omega\cap\bigcap_{l=1}^{\infty}\bigcup_{j=l}^{\infty}(\partial^* E_{i_j}\cup O_{E_{i_j}})\right)=0.
\]
To conclude the proof, we note that for any $x\in I_E\cap \Omega\setminus \bigcap_{l=1}^{\infty}\bigcup_{j=l}^{\infty}(O_{E_{i_j}}\cup \partial^* E_{i_j})$ we have $x\in I_{E_{i_j}}$ starting from some index $j\in\N$. By an analogous argument, and by passing to a further subsequence (not relabeled),
we obtain that for $\mathcal H$-almost every $x\in O_E\cap \Omega$ we have $x\in O_{E_{i_j}}$ starting from some index $j\in\N$.
\end{proof}

We note the following relation between the $1$-capacity and the Hausdorff content: for any $A\subset X$, we have
\begin{equation}\label{eq:capacity and Hausdorff content}
\capa_1(A)\le 2C_d\mathcal H_1(A),
\end{equation}
see the proof of \cite[Lemma 3.4]{KKST2}.

Besides pointwise convergence, we wish to consider uniform convergence outside sets
of small $1$-capacity. For this, we need the following lemma, whose proof is again based on a boxing inequality -type argument.

\begin{lemma}\label{lem:uniform measure theoretic interior}
Let $\Omega\subset X$ be an open set, let $E\subset X$ be a $\mu$-measurable set with
$P(E,\Omega)<\infty$, let $K\subset I_E\cap \Omega$ be a compact set, and let $\eps>0$.
Then there exists an open set $V\subset X$ with $\capa_1(V)<\eps$ such that
\[
\frac{\mu(B(x,r)\setminus E)}{\mu(B(x,r))}\to 0\quad\textrm{as }r\to 0
\]
uniformly for $x\in K\setminus V$.
\end{lemma}
\begin{proof}
Note by \eqref{eq:def of theta} that $P(E,K)=0$.
Recall the notation $K^{a}:=\{x\in X:\, \dist(x,K)<a\}$.
For each $j=2,3,\ldots$, by the fact that $P(E,\cdot)$ is a Radon measure on $\Omega$,
we can choose
\[
0<s_j<\min\{1/5,\dist(K,X\setminus \Omega)\}
\]
such that
\begin{equation}\label{eq:smallness of neighborhoods of K}
P(E,K^{s_j})<2^{-j}.
\end{equation}
Let
\[
A_j:=\left\{x\in K:\, \frac{\mu(B(x,s)\setminus E)}{\mu(B(x,s))}>\frac{1}{j}\ \ \textrm{for some }0<s<s_j/\lambda\right\}.
\]
Fix $j\ge 2$. 
Take $x\in A_j$ and $0<s<s_j/\lambda$ such that
\[
\frac{\mu(B(x,s)\setminus E)}{\mu(B(x,s))}>\frac{1}{j}.
\]
Halve the radius $s$ repeatedly and take the last number so obtained for which the above estimate holds, and call this number $r$; we find such a number by the fact that $x\in I_E$.
Then we have
\[
\frac{\mu(B(x,r)\cap E)}{\mu(B(x,r))}\ge \frac{1}{C_d}\frac{\mu(B(x,r/2)\cap E)}{\mu(B(x,r/2))}\ge \frac{1-1/j}{C_d}\ge \frac{1}{2C_d}.
\]
Thus
\begin{align*}
\min\{\mu(B(x,r)\cap E),\mu(B(x,r)\setminus E)\}
&\ge\min\left\{\frac 1j, \frac{1}{2C_d}\right\}\mu(B(x,r))\\
&\ge \frac{\mu(B(x,r))}{2C_d j}.
\end{align*}
Using these radii, we get a covering $\{B(x,\lambda r_x)\}_{x\in A_j}$ with $r_x\le s_j/\lambda$. By the $5$-covering theorem, we can extract a countable collection of disjoint balls $\{B(x_k,\lambda r_k)\}_{k=1}^{\infty}$ such that
the balls $B(x_k,5\lambda r_k)$ cover $A_j$.
Thus
\begin{align*}
\mathcal H_1(A_j)
&\le \sum_{k=1}^{\infty}\frac{\mu(B(x_k,5\lambda r_k))}{5\lambda r_k}\\
&\le C_d^{\lceil 5\lambda \rceil}\sum_{k=1}^{\infty}\frac{\mu(B(x_k,r_k))}{r_k}\\
&\le 2C_d^{\lceil 5\lambda \rceil+1}j\sum_{k=1}^{\infty}\frac{\min\{\mu(B(x,r_k)\cap E),\mu(B(x,r_k)\setminus E)\}}{r_k}\\
&\overset{\eqref{eq:relative isoperimetric inequality}}{\le} 4C_d^{\lceil 5\lambda \rceil+1} C_P j\sum_{k=1}^{\infty}P(E,B(x_k,\lambda r_k))\\
&\le 4C_d^{\lceil 5\lambda \rceil+1} C_P jP(E,K^{s_j})\\
&\le 4C_d^{\lceil 5\lambda \rceil+1} C_P j 2^{-j}
\end{align*}
by \eqref{eq:smallness of neighborhoods of K}.
Then
\[
\mathcal H_1\left(\bigcup_{j=l}^{\infty} A_j\right)\le \sum_{j=l}^{\infty}
\mathcal H_1( A_j)\le 4\sum_{j=l}^{\infty} C_d^{\lceil 5\lambda \rceil+1} C_P j 2^{-j}< \frac{\eps}{2C_d}
\]
for sufficiently large $l\in\N$, and so by \eqref{eq:capacity and Hausdorff content}
\[
\capa_1\left(\bigcup_{j=l}^{\infty} A_j\right)<\eps.
\]
By the fact that $\capa_1$ is an outer capacity, we can choose an open set $V\supset \bigcup_{j=l}^{\infty} A_j$ with $\capa_1(V)<\eps$.
By the definition of the sets $A_j$, we obtain the desired uniform convergence in
$K\setminus V$.
\end{proof}

Now we have the following uniform convergence result for sets of finite perimeter.

\begin{proposition}\label{prop:uniform convergence for sets of finite perimeter}
Let $\Omega\subset X$ be an open set, and let $E_i,E\subset X$ be $\mu$-measurable sets with  $\ch_{E_i}\to \ch_E$ in $L^1_{\loc}(\Omega)$
and $P(E_i,\Omega)\to P(E,\Omega)<\infty$.
Then there exists a subsequence (not relabeled) such that
whenever $K\subset I_E\cap\Omega$ is a compact set and $\eps>0$,
there exists an open set $U\subset X$ with $\capa_1(U)<\eps$ and an index $l\in\N$ such that 
$K\setminus U\subset I_{E_i}$ for all $i\ge l$.
\end{proposition}

\begin{proof}
Take the subsequence $\{i_j\}_{j=1}^{\infty}$ obtained in the proof of Proposition \ref{prop:strict to pointwise convergence for sets}.
Fix a compact set $K\subset I_E\cap \Omega$ and $\eps>0$.
By Lemma \ref{lem:uniform measure theoretic interior}
we find a set $V\subset X$ with $\capa_1(V)<\eps/2$ such that 
\[
\frac{\mu(B(x,r)\setminus E)}{\mu(B(x,r))}\to 0\quad\textrm{as }r\to 0
\]
uniformly for $x\in K\setminus V$. This implies in particular that for some $p\in \N$,
\[
\frac{\mu(B(x,r)\setminus E)}{\mu(B(x,r))}\le \frac{\gamma}{2}
\]
for all $x\in K\setminus V$ and $0<r<1/p$. Hence $K\setminus V\subset I_p$;
recall the definition from the beginning of the proof of Proposition \ref{prop:strict to pointwise convergence for sets}. Thus $\capa_1(K\setminus I_p)\le \capa_1(V)<\eps/2$.

Moreover, by choosing $p$ even larger, if necessary, we have $K\subset \Omega_p'$, where the sets $\Omega_k'\Subset\Omega$ were also defined in the proof of Proposition \ref{prop:strict to pointwise convergence for sets}. Now
\begin{align*}
\mathcal H_1\left(I_{p}\cap \Omega'_p\cap\bigcup_{j=l}^{\infty}(\partial^* E_{i_j}\cup O_{E_{i_j}})\right)
&\le \sum_{j=l}^{\infty}\mathcal H_1\left(I_{p}\cap \Omega'_p\cap(\partial^* E_{i_j}\cup O_{E_{i_j}})\right)\\
&<\frac{\eps}{4C_d}
\end{align*}
for a sufficiently large $l\in\N$, by the last four lines of \eqref{eq:main estimate for bad set}.
Thus by \eqref{eq:capacity and Hausdorff content},
\[
\capa_1\left(I_{p}\cap \Omega'_p\cap\bigcup_{j=l}^{\infty}(\partial^* E_{i_j}\cup O_{E_{i_j}})\right)<\frac{\eps}{2}.
\]
Then since $\capa_1(K\setminus (I_{p}\cap\Omega_p'))<\eps/2$, we conclude
\[
\capa_1\left(K\cap \bigcup_{j=l}^{\infty}(\partial^* E_{i_j}\cup O_{E_{i_j}})\right)<\eps.
\]
Since $\capa_1$ is an outer capacity, we can take an open set
\[
U\supset K\cap \bigcup_{j=l}^{\infty}(\partial^* E_{i_j}\cup O_{E_{i_j}})
\]
with $\capa_1(U)<\eps$. Now $K\setminus U\subset I_{E_{i_j}}$ for all $j\ge l$.
\end{proof}

The following lemma is essentially \cite[Exercise 1.19]{AFP}.

\begin{lemma}\label{lem:convergence}
Let $f_i,f\in L^1(\R)$ be nonnegative functions and assume that $f(t)\le \liminf_{i\to\infty}f_i(t)$ for almost every $t\in\R$, and that
\begin{equation}\label{eq:limsup estimate}
\limsup_{i\to\infty}\int_{\R}f_i\,dt\le \int_{\R}f\,dt.
\end{equation}
Then $f_i\to f$ in $L^1(\R)$.
\end{lemma}

\begin{proof}
Define $g_i:=\inf_{j\ge i}f_j$. Then for some function $g$, $g_i(t)\nearrow g(t)\ge f(t)$ for almost every $t\in\R$.
By Lebesgue's monotone convergence theorem,
\[
\int_{\R} g\,dt=\lim_{i\to\infty}\int_{\R} g_i\,dt\le \limsup_{i\to\infty}\int_{\R}f_i\,dt\le \int_{\R}f\,dt<\infty,
\]
that is, $g\in L^1(\R)$. Thus by Lebesgue's dominated convergence theorem,
$g_i\to g$ in $L^1(\R)$. But since $\lim_{i\to\infty}\int_{\R} g_i\,dt\le \int_{\R}f\,dt$, we must have $g=f$ almost everywhere.

Thus $(f_i-f)_-=(f-f_i)_+\le (f-g_i)_+\to 0$ in $L^1(X)$, and so
\begin{align*}
\limsup_{i\to\infty}\int_{\R}(f_i-f)_+\,dt
&= \limsup_{i\to\infty} \left(\int_{\R}(f_i-f)\,dt+\int_{\R}(f_i-f)_-\,dt\right)\\
&= \limsup_{i\to\infty}\int_{\R}(f_i-f)\,dt\le 0
\end{align*}
by \eqref{eq:limsup estimate}.
\end{proof}

\begin{proof}[Proof of Theorem \ref{thm:strict to pointwise convergence}]
By passing to a subsequence (not relabeled), for almost every $t\in\R$ we have $\ch_{\{u_i>t\}}\to \ch_{\{u>t\}}$ in $L^1(\Omega)$, see e.g. \cite[p. 188]{EvaG92}. Hence by lower semicontinuity, for almost every $t\in\R$
\[
P(\{u>t\},\Omega)\le \liminf_{i\to\infty} P(\{u_i>t\},\Omega).
\]
By the coarea formula \eqref{eq:coarea} and the assumption of the theorem, we have also
\[
\lim_{i\to\infty}\int_{-\infty}^{\infty} P(\{u_i>t\},\Omega)\,dt
=\int_{-\infty}^{\infty} P(\{u>t\},\Omega)\,dt.
\]
By Lemma \ref{lem:convergence} we conclude that $P(\{u_i>\cdot\},\Omega)\to P(\{u>\cdot\},\Omega)$ in $L^1(\R)$.
By passing to a subsequence (not relabeled), we have $P(\{u_i>t\},\Omega)\to P(\{u>t\},\Omega)<\infty$ for almost every $t\in\R$, and then in particular we can find a countable and dense set  $T\subset \R$ such that
\[
\ch_{\{u_i>t\}}\to \ch_{\{u>t\}}\ \ \textrm{in }L^1(\Omega)\quad\textrm{and}\quad P(\{u_i>t\},\Omega)\to P(\{u>t\},\Omega)<\infty
\]
for every $t\in T$.
If $t\in T$, by Proposition \ref{prop:strict to pointwise convergence for sets}, we can find a $\mathcal H$-negligible set $\widetilde{N}\subset X$ and a subsequence (not relabeled) such that if $x\in I_{\{u>t\}}\cap\Omega\setminus \widetilde{N}$, then $x\in I_{\{u_i>t\}}$ for sufficiently large $i\in\N$, and if $x\in O_{\{u>t\}}\cap\Omega\setminus \widetilde{N}$, then $x\in O_{\{u_i>t\}}$  for sufficiently large $i\in\N$.

By a diagonal argument, we find a $\mathcal H$-negligible set $N\subset X$ and a subsequence (not relabeled) such that if $t\in T$ and $x\in I_{\{u>t\}}\cap\Omega\setminus N$, then $x\in I_{\{u_i>t\}}$ for sufficiently large $i\in\N$, and if $x\in O_{\{u>t\}}\cap\Omega\setminus N$, then $x\in O_{\{u_i>t\}}$  for sufficiently large $i\in\N$.

By \cite[Lemma 3.2]{KKST3}, there exists a $\mathcal H$-negligible set $\widehat{N}\subset X$ such that for every $x\in \Omega\setminus \widehat{N}$, we have $-\infty<u^{\wedge}(x)\le u^{\vee}(x)<\infty$.
Fix $x\in \Omega\setminus ( N\cup \widehat{N})$. Also fix $\eps>0$. There exists
$t\in (u^{\wedge}(x)-\eps,u^{\wedge}(x))\cap T$. Now
\[
\lim_{r\to 0}\frac{\mu(B(x,r)\cap \{u\le t\})}{\mu(B(x,r))}=0,
\]
so that $x\in I_{\{u>t\}}$. Thus for sufficiently large $i\in\N$, $x\in I_{\{u_i>t\}}$, and so $u_i^{\wedge}(x)\ge t\ge u^{\wedge}(x)-\eps$. Hence
\[
\liminf_{i\to\infty} u_i^{\wedge}(x)\ge u^{\wedge}(x)-\eps.
\]
Analogously, there exists
$t\in (u^{\vee}(x),u^{\vee}(x)+\eps)\cap T$, and then $x\in O_{\{u>t\}}$. For sufficiently large $i\in\N$, $x\in O_{\{u_i>t\}}$, and thus
\[
 \limsup_{i\to\infty} u_i^{\vee}(x)\le u^{\vee}(x)+\eps.
\]
Since $\eps>0$ was arbitrary, we obtain the result.
\end{proof}

In the remainder of this section, we give our main results on uniform convergence.
Recall that a set $A\subset X$ is \emph{$1$-quasiopen} if for every $\eps>0$ there is an
open set $G\subset X$ with $\capa_1(G)<\eps$ such that $A\cup G$ is open.

\begin{proposition}[{\cite[Proposition 4.2]{L3}}]\label{prop:set of finite perimeter is quasiopen}
Let $\Omega\subset X$ be an open set and let $E\subset X$ be a $\mu$-measurable set with
$P(E,\Omega)<\infty$. Then the sets $I_E\cap\Omega$ and $O_E\cap\Omega$ are $1$-quasiopen.
\end{proposition}

\begin{lemma}[{\cite[Eq. (4.1)]{LaSh}}]\label{lem:weak L1 estimate for capacity}
If $v\in\BV(X)$ and $t>0$, then
\[
\capa_1(\{v^{\vee}>t\})\le  C_2\frac{\Vert v\Vert_{\BV(X)}}{t}
\]
for some constant $C_2=C_2(C_d,C_P,\lambda)$.
\end{lemma}

\begin{theorem}\label{thm:uniform convergence outside jump set}
Let $\Omega\subset X$ be an open set, and let $u_i,u\in \BV(\Omega)$ such that $u_i\to u$ in $L^1(\Omega)$ and $\Vert Du_i\Vert(\Omega)\to \Vert Du\Vert(\Omega)$.
Then there exists a subsequence (not relabeled) such that
whenever $K\subset \Omega\setminus S_u$ is compact and $\eps>0$, there exists
an open set $U\subset X$ with $\capa_1(U)<\eps$ and an index $l\in\N$ such that
\[
\widetilde{u}(x)-\eps\le u_i^{\wedge}(x)
\le
u_i^{\vee}(x)\le \widetilde{u}(x)+\eps.
\]
for all $x\in K\setminus U$ and $i\ge l$.
\end{theorem}

\begin{proof}
Passing to a subsequence (not relabeled) just as in the proof of Theorem
\ref{thm:strict to pointwise convergence}, we find a countable dense set $T\subset \R$
such that for all $t\in T$ we have
$\ch_{\{u_i>t\}}\to \ch_{\{u>t\}}$ in $L^1(\Omega)$ and
$P(\{u_i>t\},\Omega)\to P(\{u>t\},\Omega)<\infty$.
By a diagonal argument, pick a subsequence (not relabeled) such that the conclusion of
Proposition \ref{prop:uniform convergence for sets of finite perimeter} is true with $E$
replaced by any set
$\{u>t\}$ or $\{u\le t\}$ with $t\in T$.

Fix a compact set $K\subset \Omega\setminus S_u$, and fix $\eps>0$.
Choose $M\in\N$ as follows.
Choose a function $\eta\in \Lip_c(X)$ with $0\le \eta\le 1$, $\eta=1$ in $K$, and $\eta=0$ in $X\setminus \Omega$. Then $u\eta \in\BV(X)$ by e.g. \cite[Lemma 3.2]{HKLS}, and so by Lemma \ref{lem:weak L1 estimate for capacity}
\[
\capa_1(\{u^{\vee}>M\}\cap K)
= \capa_1(\{(u\eta)^{\vee}>M\}\cap K)
\le C_2\frac{\Vert u\eta\Vert_{\BV(X)}}{M}.
\]
By similarly estimating the set $\{u^{\wedge}<-M\}$,
we can fix a sufficiently large $M\in\N$ so that
\begin{equation}\label{eq:choice of M}
\capa_1(\{u^{\wedge}<-M\}\cap K)+\capa_1(\{u^{\vee}>M\}\cap K)+1/M<\eps/2.
\end{equation}
Then take $L\in\N$ and a strictly increasing collection of numbers $S:=\{t_j\}_{j=1,\ldots,L}\subset T$ such that
$t_{j+1}-t_j<1/M$ for all $j\in\{1,\ldots,L-1\}$, and $t_2<-M$ and $t_{L-1}>M$.

Take open sets $U_j\supset \partial^*\{u>t_j\}\cap K$ with
$\capa_1(U_j\cap U_{j+1})<\eps/4L$ for all $j\in\{1,\ldots,L-1\}$; this can
be done as follows.
For each $j$, since $P(\{u>t_j\},\Om)<\infty$, by \eqref{eq:def of theta} we have
$\mathcal H(\partial^*\{u>t_j\}\cap \Om)<\infty$, and thus we can
pick compact sets $F_j\subset \partial^*\{u>t_j\}\cap K$ such that
$\mathcal H(\partial^*\{u>t_j\}\cap K\setminus F_j)<\eps/16C_d L$. Then
$\capa_1(\partial^*\{u>t_j\}\cap K\setminus F_j)<\eps/8L$
by \eqref{eq:capacity and Hausdorff content},
and so we find
open sets $\widehat{U}_j\supset \partial^*\{u>t_j\}\cap K\setminus F_j$
with $\capa_1(\widehat{U}_j)<\eps/8L$.
Moreover, since $S_u\cap K=\emptyset$, it follows that $\partial^*\{u>t_j\}\cap \partial^*\{u>t_k\} \cap K=\emptyset$ for any $j\neq k$; this follows in a straightforward manner
from the definitions, see e.g. \cite[Proposition 5.2]{AMP}. Thus we can pick pairwise disjoint open sets $\widetilde{U}_j\supset F_j$.
Then we can define $U_j:=\widetilde{U}_j\cup \widehat{U}_j$, $j\in\{1,\ldots,L\}$.

By Proposition \ref{prop:set of finite perimeter is quasiopen},
we can also find open sets $V_j\subset X$ with $\capa_1(V_j)<\eps/8L$ such that each
$(O_{\{u>t_j\}}\cap \Om)\cup V_j$ is an open set. Thus each $((I_{\{u>t_j\}}\cup \partial^*\{u>t_j\})\cup (X\setminus\Om))\setminus V_j$ is a closed set, and so each
$(I_{\{u>t_j\}}\cup \partial^*\{u>t_j\})\cap K\setminus V_j$ is a closed set.
Then each
\[
I_{\{u>t_j\}}\cap K\setminus (U_j\cup V_j)
=(I_{\{u>t_j\}}\cup \partial^*\{u>t_j\})\cap K\setminus (U_j\cup V_j)
\]
is a compact subset of $I_{\{u>t_j\}}\cap\Omega$.
By the conclusion of Proposition \ref{prop:uniform convergence for sets of finite perimeter},
there exist open sets $W_j\subset X$ with $\capa_1(W_j)<\eps/8L$ such that
for some $l\in\N$ and all $j\in\{1,\ldots,L\}$,
\begin{equation}\label{eq:conclusion for sets in uniform convergence proof}
I_{\{u>t_j\}}\cap K\setminus (U_j\cup V_j\cup W_j)\subset I_{\{u_i>t_j\}}
\end{equation}
for all $i\ge l$.

Let 
\[
G:= K\cap \Bigg(\{u^{\wedge}<-M\}\cup \{u^{\vee}>M\}
\cup \bigcup_{j=1}^{L}(V_j\cup W_j)
\cup \bigcup_{j=1}^{L-1}(U_j\cap U_{j+1})\Bigg),
\]
so that $\capa_1(G)<\eps$. Then take an open set $U\supset G$ with $\capa_1(U)<\eps$.
Fix $x\in K\setminus U$.
For some $j\in\{2,\ldots,L-1\}$ we have
$t_j\in (\widetilde{u}(x)-1/M,\widetilde{u}(x))$. Now
\[
\lim_{r\to 0}\frac{\mu(B(x,r)\cap \{u\le t_j\})}{\mu(B(x,r))}=0,
\]
so that $x\in I_{\{u>t_j\}}$. Then clearly also $x\in I_{\{u>t_{j-1}\}}$.
Then by the definition of the set $G\niton x$, we conclude
\[
x\in \big(I_{\{u>t_j\}}\setminus (U_j\cup V_j\cup W_j)\big)\cup \big(I_{\{u>t_{j-1}\}}\setminus
(U_{j-1}\cup V_{j-1}\cup W_{j-1})\big).
\]
By \eqref{eq:conclusion for sets in uniform convergence proof},
$x\in I_{\{u_i>t_j\}}\cup I_{\{u_i>t_{j-1}\}}=I_{\{u_i>t_{j-1}\}}$ for all $i\ge l$,
and so
\[
u_i^{\wedge}(x)\ge t_{j-1}\ge t_{j}-1/M\ge \widetilde{u}(x)-2/M.
\]
Since we required $1/M<\eps/2$, we get for all $i\ge l$
\[
u_i^{\wedge}(x)\ge \widetilde{u}(x)-\eps.
\]
By making $l$ and $U$ bigger, if necessary, we get analogously
\[
 u_i^{\vee}(x)\le\widetilde{u}(x)+\eps
\]
for all $x\in K\setminus U$ and all $i\ge l$.
\end{proof}

\begin{corollary}\label{cor:uniform convergence outside jump set}
Let $\Omega\subset X$ be an open set, and let $u_i,u\in \BV(\Omega)$ such that $u_i\to u$ in $L^1(\Omega)$ and $\Vert Du_i\Vert(\Omega)\to \Vert Du\Vert(\Omega)$.
Then there exists a subsequence (not relabeled) such that
whenever $K\subset \Omega\setminus S_u$ is compact and $\eps>0$,
there exists
an open set $U\subset X$ with $\capa_1(U)<\eps$ such that
$\widetilde{u}_i\to \widetilde{u}$
uniformly in $K\setminus U$.
\end{corollary}

\section{Examples}

In Theorem \ref{thm:strict to pointwise convergence}, there are three obvious ways in which  the result could potentially be strengthened, presented in the following questions:

\begin{itemize}
\item Does the pointwise convergence hold for the original sequence, instead of a subsequence?
\item Can we obtain $u_i^{\wedge}(x)\to u^{\wedge}(x)$ and $u_i^{\vee}(x)\to u^{\vee}(x)$ for $\mathcal H$-almost every $x\in S_u$?
\item The sets where we do not obtain pointwise convergence are known to be $\mathcal H$-negligible; can we further restrict this family?
\end{itemize}

The following three examples show that the answer to each of these questions is no.

\begin{example}
Let $X=\R^2$ (unweighted) and for each $k\in\N$, let
\[
E_i:= B((2^{-k+1}(i-2^{k-1}),0),2^{-k+1}),\qquad i=2^{k-1},\ldots,2^{k}-1.
\]
Then we have
$\Vert \ch_{E_i}\Vert_{\BV(\R^2)}\to 0$ as $i\to \infty$, but for all $x=(x_1,0)$ with $0\le x_1< 1$ there exists infinitely many $i\in\N$ such that $x\in E_i\subset I_{E_i}$.
Clearly $\mathcal H(\{(x_1,x_2):\, 0\le x_1\le 1, x_2=0\})>0$; note that the codimension one Hausdorff measure $\mathcal H$ is comparable to the usual $1$-dimensional Hausdorff measure.

Thus we need to pass to a subsequence --- for example $(u_{i_j})_{j=1}^{\infty}$
with $i_j=2^{j-1}$ will do.
\end{example}

\begin{example}\label{ex:failure of convergence in jump set}
Let $X=\R$ (unweighted). Let $u:=\ch_{(0,1)}$ and
\[
u_i(x):=\max\{0, \min\{1,1/4+ix\}\}\ch_{(-\infty,1]}(x),\quad i\in\N.
\]
Clearly $u_i\to u$ in $L^1(\R)$ and $\Vert Du_i\Vert(\R)=2=\Vert Du\Vert(\R)$.
However,
\[
u_i^{\wedge}(0)=u_i^{\vee}(0)\equiv 1/4\not\to 0=u^{\wedge}(0).
\]
Similarly, we do not have convergence to $u^{\vee}(0)=1$
--- or to $\widetilde{u}(0)=1/2$.
Yet $\mathcal H(0)=2$ (note that the codimension one Hausdorff measure $\mathcal H$ is now exactly twice the usual $0$-dimensional Hausdorff measure). Thus we cannot have pointwise convergence $\mathcal H$-almost everywhere in the jump set, for any subsequence.
\end{example}

If $\Om\subset X$ is an open set and
we even have $u_i\to u$ in $\BV(\Omega)$ (that is, in the $\BV$ norm), then for a subsequence (not relabeled) we have $u_i^{\wedge}(x)\to u^{\wedge}(x)$ and $u_i^{\vee}(x)\to u^{\vee}(x)$ for $\mathcal H$-almost every $x\in \Omega$; this follows e.g. from \cite[Remark 4.1, Lemma 4.2]{LaSh}. Thus with such a stronger assumption, we do obtain
pointwise convergence even in the jump set.

\begin{example}\label{ex:failure of convergence in given negligible set}
Let $u\equiv 0$ and let $A\subset X$ be a $\mathcal H$-negligible set. For each $i\in\N$, take a covering $\{B(x_j^i,r_j^i)\}_{j\in\N}$ of $A$ such that $r^i_j\le 1$ and
\[
\sum_{j=1}^{\infty}\frac{\mu(B(x_j^i,r_j^i))}{r^i_j}<\frac{1}{i}.
\]
By applying the coarea formula \eqref{eq:coarea} with $v=d(\cdot,x_j^i)$ and $\Omega=B(x_j^i,2r_j^i)$, for each $i,j\in\N$ we find
$s^i_j\in [r^i_j,2r^i_j]$ such that
\[
P(B(x_j^i,s^i_j),X)\le 2 C_d \frac{\mu(B(x_j^i,r_j^i))}{r^i_j}.
\]
For each $i\in\N$, let $E_i:=\bigcup_{j=1}^{\infty}B(x_j^i,s_j^i)$. Then
$A\subset E_i\subset I_{E_i}$ for all $i\in\N$. However, by lower semicontinuity and
\eqref{eq:Caccioppoli sets form an algebra},
\[
P(E_i,X)\le \sum_{j=1}^{\infty} P(B(x_j^i,s_j^i),X)<\frac{2C_d}{i}\to 0\qquad\textrm{as }i\to\infty.
\]
Also
\[
\Vert \ch_{E_i}\Vert_{L^1(X)}=\mu(E_i)\le \sum_{j=1}^{\infty}\mu(B(x_j^i,s_j^i))\le
C_d\sum_{j=1}^{\infty}\frac{\mu(B(x_j^i,r_j^i))}{r^i_j}<\frac{C_d}{i}\to 0.
\]
Thus setting $u_i:=\ch_{E_i}$, the assumptions of Theorem \ref{thm:strict to pointwise convergence} are satisfied.
In conclusion, given \emph{any} $\mathcal H$-negligible set, pointwise convergence can fail at each point in the set, for all subsequences.
\end{example}

Concerning our results on uniform convergence, note first that according to Egorov's theorem, if $A\subset X$, $\nu$ is a positive Radon measure of finite mass on $A$, and $v_i,v$ are $\nu$-measurable functions on $A$ such that $v_i(x)\to v(x)$ as $i\to\infty$ for $\nu$-almost every $x\in A$, then for any $\eps>0$ there exists $D\subset A$ with $\nu(D)<\eps$ such that $v_i\to v$ uniformly in $A\setminus D$.

However, if instead of a Radon measure we work with the $1$-capacity, a problem arises from the fact that the $1$-capacity is not a Borel measure. The following example demonstrates that a Egorov-type result fails even under very favorable conditions.

\begin{example}
Let $X=\R$ and let $\Omega=(0,1)$. Let $u\equiv 0$ and
\[
u_i(x):=
\begin{cases}
ix & \textrm{for }0< x\le 1/i, \\
2-ix & \textrm{for }1/i\le x\le 2/i, \\
0 & \textrm{for }2/i\le x< 1.
\end{cases}
\]
Then $u_i,u\in C(\Omega)$ and $u_i(x)\to u(x)$ for every $x\in \Omega$.
Fix $\eps\in (0,1)$. Since $\capa_1(\{x\})=2$ for every $x\in\R$, the condition
$\capa_1(D)<\eps$ implies $D=\emptyset$. However, we do not have that $u_i\to u$
uniformly in $\Omega$.

Equally well we can consider the closed unit interval, so even with a compact set and continuous functions, things can go wrong.
It is the condition $\Vert Du_i\Vert(\Omega)\to\Vert Du\Vert(\Omega)$ that allows us
to obtain uniform convergence in Theorem \ref{thm:uniform convergence outside jump set} and Corollary \ref{cor:uniform convergence outside jump set}.
\end{example}

Recall that according to Example \ref{ex:failure of convergence in jump set},
we cannot have even pointwise convergence in the jump set. If we consider the unit interval $(0,1)\subset X\setminus S_u$ in this example, it is clear that we do not have $\widetilde{u}_i\to \widetilde{u}$ uniformly in $(0,1)$, for any subsequence.
Since again $\capa_1(\{x\})=2$ for every $x\in \R$, we see that Corollary \ref{cor:uniform convergence outside jump set} (and Theorem \ref{thm:uniform convergence outside jump set}) fail if we do not require the set $K$ to be compact.
Moreover, Example \ref{ex:failure of convergence in given negligible set}
demonstrates that we need, in general, to discard a further set of small capacity
in order to obtain uniform convergence.

\paragraph{Acknowledgments.} The research was
funded by a grant from the Finnish Cultural Foundation. The author wishes to thank Giles Shaw for posing the question that led to this research, and Jan Kristensen for discussions on capacity and uniform convergence.

\noindent Address:\\

\noindent Department of Mathematical Sciences\\
4199 French Hall West\\
University of Cincinnati\\
2815 Commons Way\\
Cincinnati, OH 45221-0025\\
E-mail: {\tt lahtipk@ucmail.uc.edu}

\end{document}